\documentclass[11pt]{article}

\usepackage{graphicx,amssymb}

\newtheorem{theorem}{Theorem}[section]
\newtheorem{corollary}[theorem]{Corollary}
\newtheorem{lemma}[theorem]{Lemma}
\newtheorem{proposition}[theorem]{Proposition}

\newtheorem{problem}[theorem]{Problem}

\begin{document}

\title{Colourings, Homomorphisms, and Partitions of Transitive Digraphs}

\author{
Tom\'{a}s Feder \\
   268 Waverley St., Palo Alto, CA 94301, USA \\
   {\tt tomas@theory.stanford.edu}\\
\\
Pavol Hell \\
   School of Computing Science \\
   Simon Fraser University \\
   Burnaby, B.C., Canada V5A 1S6 \\
   {\tt pavol@cs.sfu.ca}\\
\\
C{\'e}sar Hern{\'a}ndez Cruz \\
   Facultad de Ciencias \\
   Universidad Nacional Aut\'onoma de M\'exico \\
   Ciudad Universitaria, C.P. 04510,  M\'exico, D.F. \\
   {\tt cesar@matem.unam.mx}\\
}

\date{}


\newcommand{\DEF}[1]{{\em #1\/}}

\newcommand\chic{\chi_c}
\newcommand\C{\hbox{${\cal C}$}}
\newcommand{\RR}{\mbox{$\mathbb R$}}
\newcommand{\NN}{\mbox{$\mathbb N$}}
\newcommand{\ZZ}{\mbox{$\mathbb Z$}}
\newcommand{\eopf}{\raisebox{0.8ex}{\framebox{}}}
\newcommand{\dist}{\hbox{\rm d}}
\renewcommand\a{\alpha}
\renewcommand\b{\beta}
\renewcommand\c{\gamma}
\renewcommand\d{\delta}
\newcommand\D{\Delta}
\newcommand{\directedchi}{\mbox{$\vec{\chi}$}}
\newcommand{\directedE}{\mbox{$\vec{E}$}}
\newcommand{\directedG}{\mbox{$\vec{G}$}}
\newcommand{\directedK}{\mbox{$\vec{K}$}}

\newenvironment{proof}%
{\noindent{\bf Proof.}\ }%
{\hfill\eopf\par\bigskip}%


\maketitle

\begin{abstract}
We investigate the complexity of generalizations of colourings (acyclic colourings, $(k,\ell)$-colourings, 
homomorphisms, and matrix partitions), for the class of transitive digraphs. Even though transitive digraphs 
are nicely structured, many problems are intractable, and their complexity turns out to be difficult to classify. 
We present some motivational results and several open problems.
\end{abstract}


\section{Introduction}

Recently, there has been much interest in the complexity of colouring, homomorphism, and matrix partition
problems for restricted graph classes. These typically include the class of perfect graphs or its subclasses 
(chordal, interval, split graphs), or graphs without certain forbidden induced subgraphs \cite{paulusma,survey}.
Here we study similar questions for a restricted class of {\em digraphs}, namely, for transitive digraphs. As far 
as we know, this is the first paper concerned with the complexity of homomorphisms, or matrix partition
problems of any restricted digraph class. Transitive digraphs are known to have a very nice structure \cite{joergen}, 
and many hard problems are known to become easy when restricted to transitive digraphs \cite{jing,kernels}. 
Yet we find that even for this relatively simple class there are difficult complexity problems for colourings, 
homomorphisms, and matrix partitions. We make an initial study of these problems and identify interesting 
open questions.

Given a fixed digraph $H$, an $H$-colouring of a digraph $G$ is a homomorphism of $G$ to $H$,
i.e., a mapping $f : V(G) \to V(H)$ such that $f(u)f(v)$ is an arc of $H$ whenever $uv$ is an arc of $G$.
The $H$-{\em colouring problem} asks whether an input digraph $G$ admits an $H$-colouring. In the
{\em list} $H$-{\em colouring problem} the input $G$ comes equipped with lists $L(u), u \in V(G)$,
and the homomorphism $f$ must also satisfy $f(u) \in L(u)$ for all vertices $u$. Finally, the
$H$-{\em retraction problem} is a special case of the list $H$-colouring problem, in which each list
is either $L(u)=\{u\}$ or $L(u)=V(H)$. We also mention the more general {\em constraint satisfaction 
problem for $H$}, which is defined the same way as the $H$-colouring problem, except the structure 
$H$ is not necessarily just a digraph (i.e., a structure with one binary relation), but a structure with 
any finite number of relations of finite arities. The input $G$ is also a structure with corresponding 
relations, and the homomorphism must preserve all relations, see \cite{fv} for details.

Instead of the digraph $H$ we may consider a {\em trigraph} $H$ \cite{trig} (cf. also \cite{trig1}; in 
a trigraph, arcs (including loops) can be {\em strong} or {\em weak}. An $H$-colouring of a digraph 
$G$ in this case is a mapping $f : V(G) \to V(H)$ such that $f(u)f(v)$ is a weak arc or a strong arc 
of $H$ whenever $uv$ is an arc of $G$, and $f(u)f(v)$ is a weak arc or a non-arc of $H$ whenever
$uv$ is a non-arc of $G$. Each trigraph $H$ can also be encoded as an $m$ by $m$ matrix
$M$ (with $m=|V(H)|$), in which $M(i,j)=0$ if $ij$ is a non-arc (non-loop if $i=j$) of $H$, 
$M(i,j)=*$ if $ij$ is a weak arc (weak loop if $i=j$) of $H$, and $M(i,j)=1$ if $ij$ is a strong 
arc (strong loop if $i=j$) of $H$. We say that $M$ is the matrix {\em corresponding to} $H$.
Then an $H$-colouring of a digraph $G$ is a partition of 
$V(G)$ to parts $V_1, V_2, \dots, V_m$, so that $V_i$ is a strong clique of $G$ if $M(i,i)=1$ 
and is an independent set of $G$ if $M(i,i)=0$; and similarly there are all arcs from $V_i$ to
$V_j$ in $G$ if $M(i,j)=1$ and no arcs from $V_i$ to $V_j$ in $G$ if $M(i,j)=0$ \cite{motwa}.
Such a partition is called an $M$-{\em partition}, or simply a {\em matrix partition}. (We note
that a matrix partition problem is not directly a constraint satisfaction problem, since the
strong arc constraints are of different nature.)


If $H$ is a digraph or a trigraph, a {\em minimal $H$-obstruction} is a digraph $G$ which is 
not $H$-colourable, but for which all vertex removed subgraphs $G-v$ are $H$-colourable. 
If $H$ is a trigraph described by a matrix $M$, we also call it a {\em minimal $M$-obstruction}.
In certain cases, the number of minimal $H$-obstructions is finite: for instance when $H$ is
the trigraph with one strong loop at $a$ and one weak undirected edge $ab$, the class of graphs
that are $H$-colourable is the class of split graphs, characterized by the absence of three induced
subgraphs ($C_4, \overline{C_4}, C_5$)  \cite{gol}, and hence for this $H$ there are three minimal 
$H$-obstructions. In such a case, obviously the $H$-colouring problem is polynomial time
solvable. In other cases, the number of minimal $H$-obstructions is infinite, but the problem is
still solvable in polynomial time: for instance when $H=K_2$, the class of $H$-colourable graphs
is the class of bipartite graphs, recognizable in polynomial time, while the minimal obstructions are 
all odd cycles. Finally, in other cases, the problem is NP-complete, such as, say, when $H=K_3$.

It is worth noting that the input digraphs $G$ never have loops, whether for $H$-colouring or
$M$-partition problems. It is traditional in graph theory to define transitive digraphs in which 
there are no loops but symmetric edges are allowed. Specifically a digraph $H$ is defined to be 
{\em transitive} if for any three {\em distinct} vertices $u, v, w$, the existence of the arcs $uv, vw$ implies 
the existence of the arc $uw$. Note that an acyclic digraph is transitive if and only if the arcs define
a transitive relation in the usual sense. However, a digraph with a directed cycle is transitive if and
only if its reflexive closure (i.e., adding all loops) defines a transitive relation. This peculiarity means
that, for instance, when taking a transitive closure of a digraph we omit any loops that would exist
in a transitive closure as a binary relation.

Acyclic transitive digraphs have particularly simple structure, namely, they are exactly those digraphs 
whose reflexive closure is a partial order. It is well known \cite{joergen}, that each transitive digraph 
$G$ is obtained from an acyclic transitive digraph $J$ by {\em replication}, whereby each $j \in V(J)$ 
is replaced by $k_j \ge 0$ vertices forming a strong clique, so that if $ij$ is an arc in $J$, then all 
$k_i$ vertices replacing $i$ dominate (i.e., have arcs to) in $G$ all $k_j$ vertices replacing $j$. Note 
that all $k_j$ vertices replacing $j$ have exactly the same in- and out-neighbours in $G$ (except that 
each of them does not dominate itself). Note that the strong components of a transitive digraph $G$ 
are strong cliques. A strong component of $G$ is called {\em initial} if the corresponding vertex $j$ 
of $J$ is a source (indegree zero vertex). Also note that if $G$ is a semi-complete transitive digraph, 
then $J$ is a transitive tournament. 


A {\em digon} is a pair of arcs $uv, vu$.

\section{Homomorphisms}

Many homomorphism problems are easy when restricted to transitive input digraphs; we often
get only finitely many transitive minimal $H$-obstructions.

Suppose, for instance, that $H$ is a {\em symmetric digraph}, i.e., each arc is in a digon. 
In other words, $H$ is obtained from a graph $H'$ by replacing each edge of $H'$ by a digon.

\begin{proposition}
Let $H$ be a symmetric digraph, and let $m$ be the size of a largest strong clique in $H$.

Then all transitive minimal $H$-obstructions have $m+1$ vertices.
\end{proposition}

\begin{proof}
Since $H$ is symmetric, $G$ is $H$-colourable if and only if the underlying graph $G'$ of $G$ 
is $H'$-colourable. Since $G'$ is a comparability graph, and hence is perfect, it admits a
$k$-colouring, where $k$ is size of the largest clique in $G'$, i.e., $G'$ admits a homomorphism
to $K_k$, and contains a copy of $K_k$. Therefore $G'$ is $H'$-colourable if and only if $k \leq m$,
and hence all minimal transitive $H$-obstructions are semi-complete digraphs with $m+1$ vertices.
\end{proof}

On the other hand, consider the case when $H$ is an asymmetric digraph (i.e., has no digons). 
In this case we also have only finitely many transitive minimal $H$-obstructions. 

\begin{proposition}
Let $H$ be an asymmetric digraph, and let $m$ be the size of a largest transitive tournament in $H$.

Then the only transitive minimal $H$-obstructions are the digon and the transitive tournament on 
$m+1$ vertices.
\end{proposition}

\begin{proof}
Obviously the digon and the transitive tournament on $m+1$ vertices are transitive minimal 
$H$-obstructions. If $G$ does not contain a digon, it is an asymmetric transitive digraph. As
noted earlier, the reflexive closure of $G$ is a partial order. By Dilworth's theorem, the vertices 
of $G$ can be covered by $k$ independent sets where $k$ is the maximum length of a chain in 
this order. Therefore, the only other minimal obstruction is the transitive tournament on $m+1$ 
vertices, where $m$ is the size of the largest transitive tournament in $H$.
\end{proof}

It follows from the previous two propositions that when $H$ is a symmetric, or an asymmetric,
digraph, the $H$-colouring problem restricted to transitive digraphs is polynomial time solvable.
As we will see later (Corollary \ref{there}), the same is not expected for the case when $H$ is
transitive. However, when $H$ is transitive and also semi-complete, we do get a polynomial time
algorithm and a finite set of minimal $H$-obstructions.

\begin{proposition}\label{semic}
Suppose $H$ is a semi-complete transitive digraph.

A transitive digraph $G$ is homomorphic to $H$ if and only if every semi-complete subgraph of $G$
is homomorphic to $H$.
\end{proposition}

\begin{proof}
As noted above, $H$ is obtained by replication from a transitive tournament $T$, and $G$ by replication 
from a transitive acyclic digraph $J$. Each maximal chain in $J$ yields a semi-complete subgraph $S$ of $G$.
We prove the theorem by induction on the number of vertices of $H$. Assume we have a semi-complete transitive
digraph $H$ and a transitive digraph $G$ such that each semi-complete subgraph of $G$ admits a 
homomorphism to $H$. Note that $H$ has a unique initial strong component, say $A$, with $a$ vertices.
Let $G'$ be obtained from $G$ by removing one vertex from each initial strong component $B$ which has 
$b \leq a$ vertices. Let $H'$ be obtained from $H$ by removing one vertex $v$ from the unique initial strong
component $A$. Clearly, $H'$ is a semi-complete transitive digraph, and $G'$ a transitive digraph; we claim
that each semi-complete subgraph of $G'$ is homomorphic to $H'$. Indeed, suppose $S$ is a semi-complete
subgraph  of $G'$. Since $S$ admits a homomorphism $f$ to $H$, $f$ is also a homomorphism to $H'$ unless
it takes  some $s \in S$ to $v$. It is easy to see that this implies that $S$ intersects one of the initial strong
components of size $b \le a$. Since a vertex was removed from this component, we have $S$ homomorphic
to $H'$. By induction $G'$ admits a homomorphism to $H'$ and so $G$ admits a homomorphism to $H$.
(Note that $v$ dominates all vertices of $H$.)
\end{proof}

\begin{corollary}
If $H$ is a semi-complete transitive digraph with $m$ vertices, then all the minimal transitive obstructions 
are semi-complete digraphs with at most $m+1$ vertices.
\end{corollary}

We show below that we need $H$ to be both transitive and semi-complete to have finite number of minimal
transitive obstructions.

Up to this point, we have always seen polynomial algorithms and finite sets of minimal obstructions go
hand in hand, for transitive inputs. There are nevertheless digraphs $H$ that have infinitely many transitive 
minimal obstructions, but can be solved in polynomial time for transitive inputs.

\begin{figure}
\begin{center}
\includegraphics{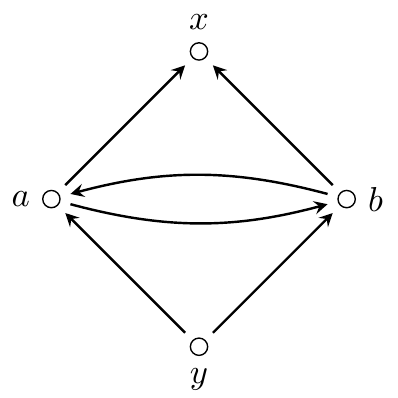}
\end{center}\caption{The digraph $H$} \label{digH}
\end{figure}

\begin{proposition}
The digraph $H$ in Figure \ref{digH} has infinitely many minimal $H$-obstructions.

However, the $H$-colouring problem is polynomial-time solvable.
\end{proposition}

\begin{figure}
\begin{center}
\includegraphics{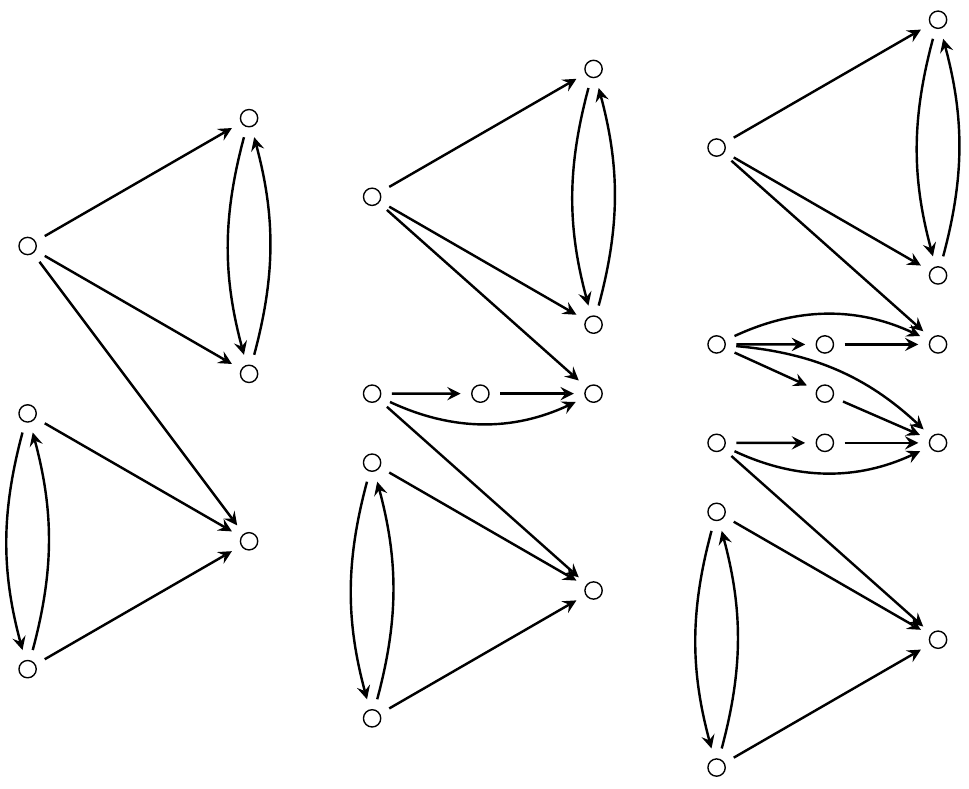}
\end{center}\caption{An infinite family of transitive minimal $H$-obstructions} \label{inffam}
\end{figure}

\begin{proof}
All the digraphs in Figure \ref{inffam} are easily checked to be transitive minimal $H$-obstructions.

We show that the $H$-colouring problem can be solved in polynomial time, by reduction to 2-SAT.
We assume that $G$ has no strong clique of size three or semi-complete digraph of size four, otherwise
$G$ is definitely not $H$-colourable. Moreover, we may assume that $G$ has no isolated vertices. Now
each strong component has at most two vertices. We remove one vertex from each digon, obtaining an
acyclic transitive digraph $G'$, in which each vertex is either a source (in-degree zero), or a sink (out-degree 
zero), or an "intermediate" vertex, i.e., a vertex with at least one in-neighhbour and at least one out-neighbour
(where all the in-neighbours are sources and all out-neighbours are sinks). Note that the absence of semi-complete 
digraphs of size four implies that no "intermediate" vertex is the result of a digon contraction, and there are no arcs 
between two "intermediate" vertices. We associate a variable $x_w$ with each source and sink $w$ in $G'$. The
intended meaning for $x_w=0$ will be that $w$ maps to $a$ or $b$, and $x_w=1$ will mean $w$ maps to $x$
if $w$ is a sink, and maps to $y$ if $w$ is a source. With this in mind, we set up a constraint for each arc 
$uv$ in $G'$ where $u$ is a source and $v$ is a sink as follows.

\begin{itemize}
\item
If $w$ is a contracted digon, then we add the constraint $x_w=0$.
\item
If neither $u$ nor $v$ is  a contracted digon then we add the constraint 
$\overline{x_u} \vee \overline{x_v}$.
\item
If there is an "intermediate" vertex $w$ that dominates $v$ and is dominated by $u$, in the original digraph $G$, 
then we add the constraint $x_u \vee x_v$.
\end{itemize}

Note that the third condition implies that if one of $u, v$ is a contracted digon, then the variable of the other vertex
obtains the value $1$.

If $G$ is homomorphic to $H$, then so is $G'$, and setting the variable values as suggested by their intended
meaning will clearly satisfy all constraints. On the other hand, if we have a satisfying assignment, we may
define a homomorphism of $G$ to $H$ as follows. Any vertex $w$ with $x_w=1$ will map to $x$ if it
is a sink and to $y$ if it is a source. It is easy to see that the set of vertices $w$ that are either ``intermediate''
vertices or have $x_w=0$ induces a bipartite subgraph of $G$ (because of the third set of constraints).
Therefore all the remaining vertices can be mapped to the digon $ab$.
\end{proof}

The minimal obstructions to 2-SAT are well understood \cite{papa}. It is not difficult to conclude from these
that all the minimal obstructions for this digraph $H$ are, in addition to the strong clique with three vertices
and all transitive semi-complete digraphs on four vertices (not containing a strong clique on three vertices),
only digraphs containing as spanning subgraphs the digraphs in Figure \ref{inffam}.

We note that the same infinite family in Figure \ref{inffam} is also a family of minimal $H^*$-obstructions 
for the semi-complete digraph $H^*$ obtained from $H$ by adding the arc $xy$. This shows that we 
cannot weaken Proposition \ref{semic} by removing the assumption that $H$ is transitive. (Corollary
\ref{there} below shows that we can not remove the other assumption, that $H$ be semi-complete.)

\vspace{2mm}

Next we note that there are many digraphs $H$ that yield NP-complete $H$-colouring problem even 
when restricted to transitive digraphs.

We consider the following construction. Let $H$ be a bipartite graph with at most $n$ black 
and at most $n$ white vertices. We form the digraph $H'$ as follows. We first orient all edges 
of $H$ from the white vertices to the black vertices. Let $P_i$ be a directed path with $n+2$
vertices, in which the first, and the $i+1$-st, vertex have been duplicated (as described in the 
introduction). Let $R_i$ also be a directed path with $n+2$ vertices, in which the the last, 
and the $(i+1)$-st, vertex have been duplicated. We identify the last vertex of each $P_i$ with
the $i$-th white vertex (if any) of $H$ and the first vertex of each $R_i$ with the $i$-th black
vertex (if any) of $H$. Then $H'$ is obtained from the resulting digraph by taking the transitive
closure. It is easy to see that the added paths ensure that the only homomorphism of $H'$ to
itself is the identity. Also consider a directed path $P$ with $n+2$ vertices with only the first
vertex duplicated, and a directed path $R$ with $n+2$ vertices and only the last vertex duplicated.
Note that $P$ admits a homomorphism to each $P_i$ and $R$ admits a homomorphism to each
$R_i$. For future reference, we define the {\em level} of the $j$-th vertex of $P$ or $P_i$ to be
$j$, and the {\em level} of the $j$-th vertex of $R$ or $R_i$ to be $n+2+j$; in this we assume
the duplicated vertices to have the same level. Note that this makes all white vertices to have
level $n+2$ and all black vertices to have level $n+3$.

\begin{figure}
\begin{center}
\includegraphics{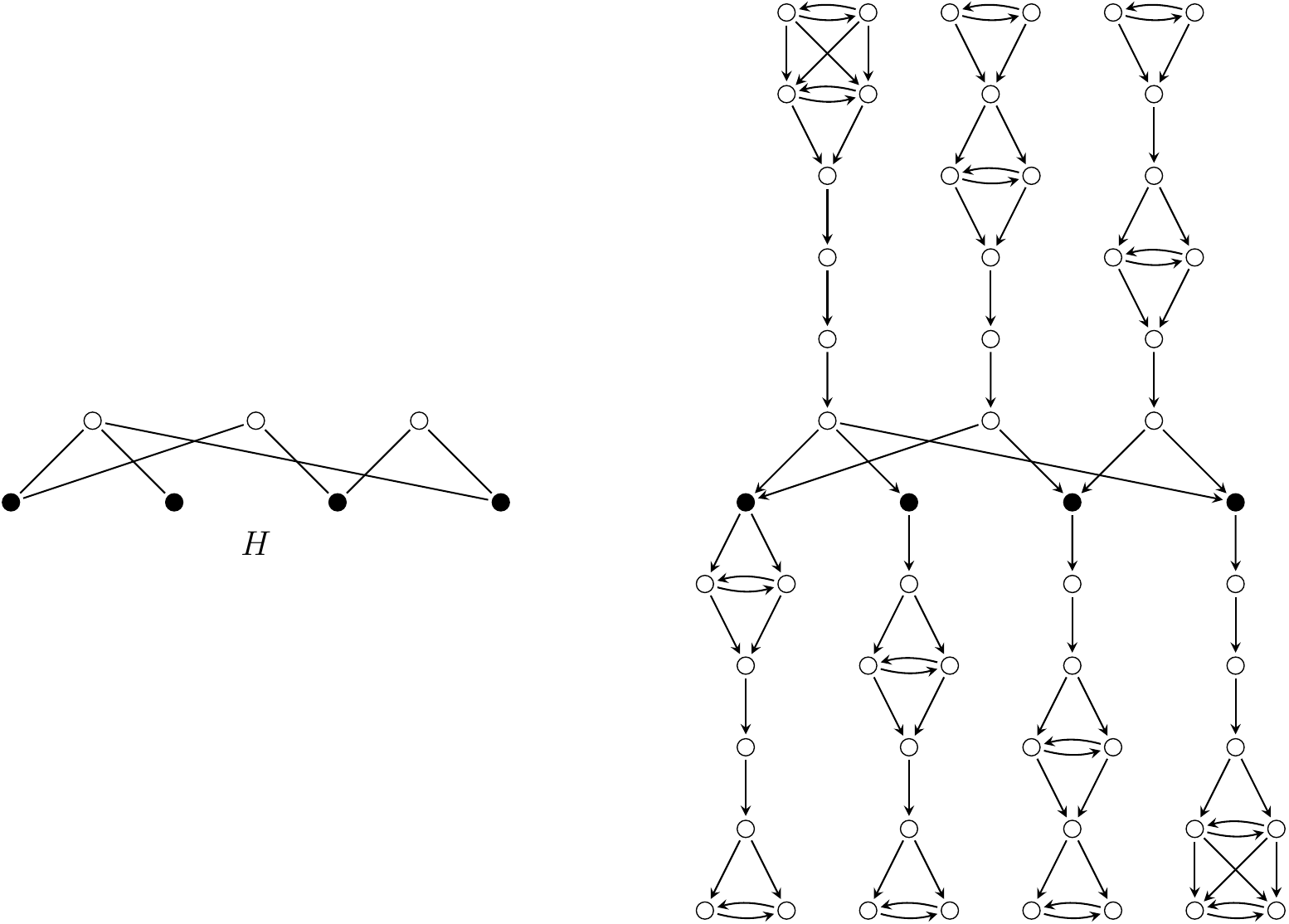}
\end{center}\caption{Example of the construction of $H'$ (without the transitive closure)} \label{constr}
\end{figure}

\begin{theorem}\label{here}
If $H$ is a bipartite graph such that the $H$-retraction problem is NP-complete, then the 
$H'$-colouring problem is NP-complete, even when restricted to transitive digraphs.
\end{theorem}

\begin {proof} 
Suppose $G$ is an instance of the $H$-retraction problem, i.e., a bipartite graph 
containing $H$ as a subgraph with lists $\{x\}$ for each (black and white) vertex $x$ of $H$, 
and lists $V(H)$ for all other  (black and white) vertices of $G$. We construct an instance $G'$ 
of the $H'$-colouring problem by orienting all edges of $G$ from the white vertices to the black 
vertices, attaching paths $P_i$ and $R_j$ to the vertices of $H$ as in the construction of $H'$,
and then (for the vertices not in $H$) we identify the last vertex of a (separate) copy of $P$ to 
each white vertex of $G$ not in $H$, and identify the first vertex of a (separate) copy of $R$ to 
each black vertex of $G$ not in $H$, and finally we take a transitive closure. Now it is easy to 
see that each homomorphism of $G'$ to $H'$ preserves the level of vertices, and that $G'$ 
admits an $H'$-colouring if and only if $G$ admits a retraction to $H$.
\end{proof}

We now observe that the construction of $H'$ ensures that it is itself transitive. Thus we have
the following fact.

\begin{corollary}\label{there}
There exists a transitive digraph $H'$ such that the $H'$-colouring problem is NP-complete
even when restricted to transitive digraphs.
\end{corollary}

Note that each $H$-colouring problem is a matrix partition problem, for a matrix $M$ 
without $1$'s. (If we view all arcs of a digraph to be weak arcs of a trigraph, we obtain the same problem.) 
In the next section, we shall see that more general matrix partition problems exhibit behaviour as complex 
as any constraint satisfaction problem.

\section{Matrix Partitions}


One of the foremost open problems in this area is the conjecture of Feder and Vardi, which asserts
that for each relational structure $H$ (as defined above) the constraint satisfaction problem to $H$
is polynomial time solvable or NP-complete \cite{fv}. It was shown in \cite{fv} that for any such
relational structure $H$ there exists a bipartite graph $H'$ such the constraint satisfaction problem
for $H$ is polynomially equivalent to the $H'$-retraction problem. This implies that classifying the
complexity of bipartite retraction problems is at least as hard as classifying the complexity of all 
constraint satisfaction problems (i.e., very hard). We prove that a similar result holds also for
matrix partition problems of transitive digraphs.

\begin{theorem}\label{csp}
For any relational structure $H$, there exists a matrix $M$ such that the constraint satisfaction 
problem for $H$ is polynomially equivalent to the $M$-partition problem for transitive digraphs.
\end{theorem}

\begin{proof}
According to \cite{fv}, we may start with the $H$-retraction problem for a bipartite graph $H$.
It also follows from \cite{fv} that we may assume that the white vertices of $H$ are domination-free, 
i.e., that their black neighbourhoods are free of inclusion. In that case the $H$-retraction problem 
may be assumed to have singleton lists $\{x\}$ only on the black vertices $x$ of $H$, with all other 
vertices of $G$ (the input graph containing $H$) having lists $V(H)$ (see \cite{fv}).

We now describe how to transform this bipartite graph $H$ into a trigraph $H'$ such that the 
above $H$-retraction problem is polynomially equivalent to the $H'$-colouring problem. (The
matrix $M$ for the theorem will then be the matrix describing the trigraph $H'$.)

First, we orient all edges of $H$ from white to black vertices, and make them weak arcs of $H'$.
Now replace each black vertex $v$ of $H$ by three vertices $v_1, v_2, v_3$ and two strong arcs
$v_1v_2, v_2v_1$. All other (white) vertices of $H$ will have the same (weak) arcs to each $v_i$ 
as they had to $v$. Let $M$ be the matrix corresponding to $H'$.

Suppose $G$ is an instance of the $H$-retraction problem, i.e., a bipartite graph containing $H$, 
with each black vertex $x$ of $H$ having the list $\{x\}$; all other vertices, including the white 
vertices of $H$, have the list $V(H)$. Then construct a transitive digraph $G'$ by orienting all 
edges of $G$ from white to black vertices, and duplicating in $G$ each black vertex $x$ of $H$
into a digon $x_1x_2, x_2x_1$. (This is a replicating operation as we have introduced it earlier.)
We now claim that $G$ admits a list $H$-colouring (with these lists, i.e., fixing all black vertices 
of $H$) if and only if $G'$ admits an $H'$-colouring, i.e., an $M$-partition. Indeed, if $G$ has 
an $H$-colouring $f$ with $f(x)=x$ for each black vertex $x$ of $H$, then $f(y)=y$ also for 
each white vertex $y$ of $H$, since $H$ is domination-free. We may assume that $f$ maps
white vertices of $G$ to white vertices of $H$ and black vertices of $G$ to black vertices of $H$,
and define an $H'$-colouring $f'$ of $G'$ (where $H'$ is the trigraph described above), by setting 

\begin{itemize}
\item
$f'(x)=f(x)$ if $x$ is a white vertex of $G$
\item
$f'(x)=v_3$ if $x$ is a black vertex of $G$ not in $H$ and $f(x)=v$ is a black vertex of $H$
\item
$f'(x_1)=x_1, f'(x_2)=x_2$ if $x$ is a black vertex of $H$ (in $G$).
\end{itemize}

It is easy to verify that $f'$ is indeed an $H'$-colouring of $G'$. Conversely, suppose that the
digraph $G'$ admits an $H'$-colouring. Note that each digon in the digraph $G'$ must map to 
a digon in the trigraph $H'$; since all digons in $H'$ have strong edges, it is easy to see that 
different digons in $G'$ must map to different digons of $H'$. (The preimage of a strong digon
must be a complete bipartite graph consisting of digons in $G'$.) This yields a homomorphism 
$f$ of $G$ to $H$ in which the subgraph $H$ of $G$ maps bijectively to $H$ (here we use again 
the fact that $H$ is domination-free). It is easy to see it corresponds to an automorphism of $H$. 
By permuting the images of $f$ on $H$ (by the inverse of the above automorphism) we obtain an 
$H$-retraction of $G$ as required. This proves the above claim, and establishes a polynomial 
(Karp) reduction from the $H$-retraction problem for bipartite graphs to the $H'$-colouring, 
i.e., $M$-partition, problem, for transitive digraphs.

We now establish a Turing reduction from the $H'$-colouring problem for transitive digraphs to 
polynomially many $H$-retraction problems, with varying lists on $G$. Thus suppose $G'$ is a 
transitive digraph instance of the $H'$-colouring problem. Because of the transitivity of $G'$, its 
digons must form disjoint strong cliques, each of which must map to a distinct digon of $H'$,
as before. For each assignment of the strong cliques of $G'$ to the strong digons of $H'$, we
take $G'$, identify each strong clique $C$ to a vertex, and call the resulting graph $G$. The new
vertex will have a list $\{x\}$, where $C$ was assigned to the digon $x_1, x_2$ in $H'$. Then it
is clear that $G'$ is $H'$-colourable by a colouring corresponding to the given assignment, if and
only if $G$ admits an $H$-retraction. Since $H$ is fixed, there are only polynomially many such
assignments of strong cliques of $G'$ to the digons of $H'$, and $G'$ is $H'$-colourable if and
only if for at least one of these assignments $G$ is $H$-colourable.
\end{proof}

We note that the matrix $M$ in the above construction has zero diagonal. (There are no loops in
$H'$.) By contrast, we will show below that when $M$ has a diagonal of $1$'s, the problem is 
always polynomial, and in fact has only finitely many transitive minimal $M$-obstructions.

\begin{theorem} \label{M1}
Let $M$ be an $m$ by $m$ matrix that has only $1$'s on the main diagonal.   Then every 
transitive minimal $M$-obstruction is an acyclic digraph with at most $m+1$ vertices.
\end{theorem}

\begin{proof}
Let $G$ be a transitive digraph that is not $M$-partionable. If $G$ is not acyclic, then $G$ has 
a proper subgraph obtained by removing a vertex from a non-trivial strong clique, which is
clearly also no $M$-partitionable. Hence, $G$ is not a minimal $M$-obstruction. If $G$ is 
acyclic but has more than $m+1$ vertices, then we can choose a proper subset $X$ of
$m+1$ vertices.  The digraph $G'$ induced on these vertices must not be $M$-partitionable,
since $G'$ has only strong cliques consisting of one vertex each, and an $M$-partition is a
partition into $m$ strong cliques. Thus again $G$ is not a minimal $M$-obstruction.
\end{proof}

We don't know whether a result analogous to Theorem \ref{csp} holds for $H$-colouring by 
digraphs $H$. In other words, we don't know whether for any relational structure $H$ there
exists a digraph $H^*$ such that the constraint satisfaction problem for $H$ is polynomially 
equivalent to the $H^*$-colouring problem for transitive digraphs. Recalling that since $H^*$ 
is a digraph, the $H^*$-colouring problem is an $M$-partition problem with a matrix $M$ 
without $1$'s. So we formulate the following open problem.

\begin{problem}
Does there exist, for any relational structure $H$, a matrix $M$ without $1$'s, such that the 
constraint satisfaction problem for $H$ is polynomially equivalent to the $M$-partition problem 
for transitive digraphs?
\end{problem}

We do have the following interesting intermediate result. 

\begin{theorem}
For any relational structure $H$, there exists a matrix $M$ with no $1$'s off the main diagonal,
such that the constraint satisfaction problem for $H$ is polynomially equivalent to the 
$M$-partition problem for transitive digraphs.
\end{theorem}

\begin{proof}
The proof is similar to the above proof of Theorem \ref{csp}. As before, we work with an
$H$-retraction problem for a bipartite graph $H$ in which the white vertices of $H$ are 
domination-free and the $H$-retraction problem is assumed to have singleton lists $\{x\}$ 
only on the black vertices $x$ of $H$. This bipartite graph $H$ is transformed into a trigraph 
$H'$ such that the $H$-retraction problem is polynomially equivalent to the $H'$-colouring 
problem. First, all edges of $H$ are oriented from white to black vertices, and made into weak 
arcs of $H'$. Second, each black vertex $v$ of $H$ is replaced by two vertices $v_1, v_2$ 
and one strong loop on $v_2$. All other vertices of $H$ have the same arcs to each $v_i$ as 
they had to $v$. If $G$ is a bipartite graph containing $H$ with each black vertex $x$ of $H$ 
having the list $\{x\}$, then construct a transitive digraph $G'$ by orienting all edges of $G$ 
from white to black vertices, and duplicating in $G$ each black vertex $x$ of $H$ into a digon 
$x_1x_2, x_2x_1$. The rest of the proofs are almost identical.
\end{proof}

\section{Colourings}

An interesting special case of matrix partition occurs when $M$ has only asterisks off the main
diagonal. If there are $k$ zeros and $\ell$ ones on the main diagonal, the $M$-partition problem
asks whether the input digraph $G$ can be partitioned into $k$ independent sets and $\ell$ cliques.
Such a partition is usually called a $(k,\ell)$-colouring of $G$ \cite{gimbel}. A minimal
$M$-obstruction in this case will be called a {\em minimal $(k,\ell)$-obstruction}.

Consider first the $(k,0)$-colouring problem, which is just the classical $k$-colouring problem.

\begin{proposition}\label{p1}
Every transitive minimal $(k,0)$-obstruction is a transitive semi-complete digraph with $k+1$ 
vertices.
\end{proposition}

\begin{proof}
The directions of the arcs again do not matter in this case, and we may just look at the underlying
undirected graph. So suppose $G$ is a transitive minimal $(k,0)$-obstruction. Then the 
underlying graph $G'$ of $G$ is again a comparability graph, and hence perfect; therefore $G'$
is the clique $K_{k+1}$. It follows that $G$ must be a transitive semi-complete digraph with $k+1$ 
vertices.
\end{proof}

On the other hand, when $k=0$, we have a matrix partition with the matrix $M$ having ones on 
the main diagonal. Using the result of Theorem \ref{M1}, we conclude the following fact.

\begin{corollary}\label{p2}
Every transitive minimal $(0,\ell)$-obstruction has at most $\ell + 1$ vertices. 
\end{corollary}

Consider next the case when $k=1$. Note that if $G$ is an acyclic transitive digraph, then it has
a $(1,\ell)$-colouring if and only if it has a vertex cover of size $\ell$. Thus if $G$ is a minimal 
acyclic transitive obstruction to $(1,\ell)$-colouring, then it is a minimal obstruction also to
having a vertex cover of size $\ell$. According to \cite{dinneen1,dinneen2}, this implies that
$G$ has at most $2\ell + 2$ vertices. (The first paper implies this for minor obstructions, the
second paper proves that they are in fact subgraph obstructions.)

For $G$ not necessarily acyclic, with some additional lemmas,  we can still obtain the same bound 
on the size of a minimal obstruction.

\begin{lemma} \label{maxcomp}
Suppose $G$ is a transitive minimal ($k,\ell$)-obstruction.
A strong clique of $G$ has at most $k+1$ vertices.
\end{lemma}

\begin{proof}
If $\ell = 0$, then every minimal obstruction is a semi-complete transitive digraph on $k+1$ vertices,
so the result holds. Thus, we will assume $\ell \ge 1$.
	
Suppose that $C$ is a strong clique of $G$ with at least $k+2$ vertices and let $v$ be one such vertex.   
Since $G$ is a minimal obstruction, $G - v$ has an $M$-partition $(V_1, \dots, V_m)$.   Recalling that
$C$ is  a strong clique of $G$ and $V_j$ is an independent set for $1 \le j \le k$, it follows that there is a
vertex $u \in C \setminus \{ v \}$ such that $u \in V_i$ for  $k+1 \le i \le m$.   By the choice of $i$, the
set $V_i$ is a strong clique of $G$. Also, $u, v$ form a digon in $G$, so by the transitivity of $G$, we know
that the set $V_i \cup v$ is also a strong clique  of $G$.   Therefore, $(V_1, \dots, V_i \cup \{ v \},
\dots, V_m )$ is a ($k,\ell$)-colouring of $G$, a contradiction.
\end{proof}

\begin{lemma} \label{cliqueelimination}
Suppose $G$ is a transitive minimal ($k,\ell$)-obstruction. If $G$ has $r \le \ell$ strong
cliques $C_1, \dots, C_r$, each with $k+1$ vertices, then $G - \bigcup_{i = 1}^r C_i$ is 
a minimal ($k,\ell-r$)-obstruction.
\end{lemma}

\begin{proof}
By induction on $r$.   The case for $r = 0$ is trivial.   Suppose $G' = G - \bigcup_{i = 1}^r C_i$ is a minimal
($k, \ell-r$)-obstruction, and let $C_{r+1}$ be a strong clique on $k+1$ vertices of $G'$.   Observe that 
$G' - C_{r+1}$ is not ($k, \ell-(r+1)$)-colourable, otherwise $G'$ would be ($k, \ell-r$)-colourable.

For a contradiction, suppose that $G' - C_{r+1}$ is not a minimal ($k, \ell-(r+1)$)-obstruction.
Hence, $G' - C_{r+1}$ contains a minimal ($k, \ell-(r+1)$)-obstruction $B$. Since $G'$ is a minimal 
($k, \ell-r$)-obstruction, the subgraph induced by $V(B) \cup C_{r+1}$ has a ($k, \ell-r$)-colouring.   
Let $(V_1, \dots, V_{k+\ell-r})$ be one such colouring.   At least one vertex of $C_{r+1}$ must belong 
to some $V_j$ with $k+1 \le j$, say $j = k+\ell-r$, because $V_i$ is an independent set for every 
$1 \le i \le k$ and $C_{r+1}$ is a strong clique.   Thus, $(V'_1, \dots, V'_{k+\ell-r})$ is also a 
($k, \ell-r$)-colouring of the subgraph induced by $V(B) \cup C_{r+1}$, where 
$V'_i = V_i \setminus C_{r+1}$, for $i <k+\ell-r$ and $V'_{k+\ell-r} = C_{r+1}$.   
But clearly $(V'_1, \dots, V'_{k+\ell-(r+1)})$ is a ($k,\ell-(r+1)$)-colouring of $B$, a contradiction.
\end{proof}

\begin{theorem}\label{p3}
Suppose $G$ is a transitive minimal $(1,\ell)$-obstruction.   Then $G$ has at most $2(\ell+1)$-vertices.
\end{theorem}

\begin{proof}
Suppose that $G$ is a transitive minimal ($1,\ell$)-obtruction.   It follows from Lemma \ref{maxcomp} that
every largest clique of $G$ has two vertices, i.e., it is a digon, and it follows from the transitivity of $G$
that the digons of $G$ are pairwise disjoint.   Let $C_1, \dots, C_r$ be the set of digons of $G$.   Suppose
that $r \le \ell$, otherwise consider only the first $\ell$ digons in the sequence $C_1, \dots, C_r$.   Let $G'$
be the transitive digraph $G' = G - \bigcup_{i=1}^r C_i$.   By Lemma \ref{cliqueelimination} the digraph
$G'$ is a minimal ($1,\ell-r$)-obstruction.   If $\ell-r = 0$, then the underlying graph of $G'$ is a minimal
non-$1$-colourable comparability graph, and hence it has exactly $2$ vertices.   Otherwise, $G'$ is an
acyclic transitive minimal ($1,\ell-r$)-obstruction.   As described above, this problem corresponds to
finding a vertex cover of size at most $\ell-r$, and the obstructions for such problem have at most
$2(\ell-r + 1)$ vertices.   In either case, since we removed $2 r$ vertices from $G$, adding up the vertices
remaining in $G'$ we obtain a bound of $2(\ell+1)$ for the vertices of $G$.
\end{proof}

Let $M$ be a general matrix with $k$ is zeros and $\ell$ $1$'s on the main diagonal.
It is an interesting phenomenon that many matrix partition problems which have only a finite
number of minimal $M$-obstructions have the size of minimal $M$-obstructions bounded 
by $(k+1)(\ell+1)$ \cite{perfect,full,survey}.

\begin{problem}
Show that there are only finitely many transitive minimal $(k,\ell)$-obstructions.

Is it true that all transitive minimal $(k,\ell)$-obstructions have at most $(k+1)(\ell+1)$ vertices?
\end{problem}

Note that the bound of $(k+1)(\ell+1)$ would unify all the previous cases discussed above, Proposition
\ref{p1}, Corollary \ref{p2}, and Theorem \ref{p3}.

We complement the above problem with a polynomial time algorithm for ($k,\ell$)-colourability 
of transitive digraphs.

\begin{theorem} \label{algor}
A transitive digraph $G$ is ($k,\ell$)-colourable if and only if the removal of some $\ell$ strong
components of $G$ yields a $k$-colourable underlying graph.
\end{theorem}

\begin{proof}
If the removal of some $\ell$ strong components yields a $k$-colourable graph, then $G$ is ($k,\ell$)-colourable, 
since the strong components of $G$ are strong cliques. On the other hand, the disjoint strong cliques in any
($k,\ell$)-colouring of $G$ may be enlarged to strong components of $G$.
\end{proof}

Theorem \ref{algor} yields the following polynomial time algorithm to check whether a transitive digraph $G$ 
has a ($k,\ell$)-colouring.   Let $\mathcal{S}$ be the set of strong components of $G$. If $|\mathcal{S}| \le
\ell$, then $G$ is ($k,\ell$)-colourable.   Otherwise, for each subset $S$ of $\mathcal{S}$ such that
$|S| = \ell$, we check whether the underlying graph of $G - S$ is $k$-colourable; this can be done in
polynomial time, since it is a comparability graph. There are only ${|\mathcal{S}| \choose \ell} \le {|V| \choose \ell}$
subsets of $\mathcal{S}$ of size $\ell$, hence the overall time of performing the aforementioned
test for all such subsets of $\mathcal{S}$ is polynomial.

\vspace{2mm}

As a last observation, we note that there is another popular notion of digraph colouring, namely
acyclic colouring \cite{victor}. Specifically, an {\em acyclic $k$-colouring} of a digraph $G$ is a partition
of the vertices into $k$ parts none of which contains a cycle of $G$; and the {\em dichromatic number},
more recently just called the {\em chromatic number} \cite{keevash}, of a digraph $G$ is the smallest 
integer $k$ such that $G$ has an acyclic $k$-colouring. An acyclic transitive digraph $J$ clearly has 
chromatic number equal to one, and it follows from the description of the structure of an arbitrary
transitive digraph discussed at the end of Section 1, that the chromatic number of an arbitrary transitive
digraph $G$, obtained from an acyclic transitive $J$ by vertex substitutions, is equal to the maximum
value $k_j$ of the size of any replacing set of vertices. Thus the only transitive minimal obstruction to 
$k$-colouring is the strong clique on $k+1$ vertices.

\section{Conclusions}

For colouring problems on transitive digraphs we've observed that acyclic $k$-colouring has a unique
minimal transitive obstruction (the strong $(k+1)$-clique), and hence can be decided in polynomial
time. On the other hand, for $(k,\ell)$-colourings there are many interesting open problems, with
the most natural question being whether for any $k$ and $\ell$ all minimal transitive $(k,\ell)$-
obstructions have at most $(k+1)(\ell+1)$ vertices. We were able to confirm that this is the case in
the cases $k=0, \ell=0,$ and $\ell=1$. In general we don't even know if the number of minimal
transitive obstructions is finite, or if the problem is solvable in polynomial time for transitive digraphs.

For homomorphism, i.e., $H$-colouring, problems on transitive digraphs, we have shown that there
are only finitely many transitive minimal obstructions (and hence the problems are solvable in polynomial
time for transitive digraphs) in the cases when $H$ is symmetric, or asymmetric, or transitive and
semi-complete. (In the latter case, both transitivity and semi-completeness have been shown to be
necessary for the conclusion.) We have seen the full range of possibilities for the difficulty of $H$-colouring 
problems for transitive digraphs $H$, i.e., finitely many transitive minimal $H$-obstructions, infinitely 
many transitive minimal $H$-obstructions but polynomial $H$-colouring problem for transitive digraphs, 
and NP-complete $H$-colouring problem, even when restricted to transitive digraphs.

For matrix partition, i.e., $M$-partition, problems on transitive digraphs, we have seen that the range
of possibilities is as large as that for constraint satisfaction problems, since for ever template $H$ for
a constraint satisfaction problem there exists a matrix $M$ such that the $H$-colouring problem is
polynomially equivalent to the $M$-partition problem for transitive digraphs. In this, we were able to 
ensure that the matrix $M$ has zero diagonal, or no $1$'s off the main diagonal; however, when the 
main diagonal of $M$ consists of $1$'s, the problem becomes easy, as there are only finitely many 
transitive minimal $M$-obstructions.

We have also identified a number of interesting questions to explore.

\end{document}